\documentclass[11pt]{article}

\usepackage{amssymb}
\usepackage{color,amsmath,amssymb, amsfonts,amstext,amsthm}

\usepackage{epsfig, graphicx, graphics}
\usepackage{longtable}

\newcommand{\p}{\partial}

\renewcommand{\phi}{\varphi}

\renewcommand{\a}{\alpha}

\newcommand{\R}{{\mathbb R}}

\newtheorem{example}{Example}
\newtheorem{theorem}{Theorem}

\newtheorem{remark}{Remark}

\title{Dynamical Inference for Transitions in  Stochastic Systems with $\alpha-$stable L\'evy Noise
 \footnote{This work was partly supported by the NSF Grant  1025422. } }

\author{Ting Gao, Jinqiao Duan \& Xingye Kan   \\
\\
 Department of Applied Mathematics, Illinois Institute of Technology \\
  Chicago, IL 60616, USA \\
 \emph{E-mail:  tinggao0716@gmail.com,  duan@iit.edu }\\
   \&\\
School of Mathematics, University of Minnesota\\
Minneapolis,  MN 55414, USA\\
\emph{E-mail: xkan@umn.edu     }  }

\begin{document}
\date\today

\maketitle

\pagestyle{plain}

\begin{abstract}
 A goal of data assimilation is to infer   stochastic dynamical behaviors  with available   observations.  We consider transition phenomena between metastable states for a stochastic system with (non-Gaussian) $\alpha-$stable L\'evy noise.
 With either discrete time or continuous time observations, we infer such transitions by computing the corresponding nonlocal Zakai equation (and its discrete time counterpart) and examining the most probable orbits for the state system. Examples are presented to demonstrate this approach.

\medskip

 {\bf Short Title:}   Transitions in Non-Gaussian Stochastic Systems  \\

 {\bf Key Words:}   Nonlocal Zakai equation; nonlocal Laplace operator;
 non-Gaussian noise; transitions between metastable states; most probable orbits


\textbf{PACS} (2010):   05.40.Ca, 02.50.Fz, 05.40.Fb, 05.40.Jc

\end{abstract}

\section{Introduction}  \label{intro}

Random fluctuations in  nonlinear systems in engineering and science
are often  non-Gaussian \cite{Woy}. For instance, it has been argued
that diffusion by geophysical turbulence \cite{Shlesinger}
corresponds to a series of  ``pauses", when the
particle is trapped by a coherent structure, and ``flights" or
``jumps" or other extreme events, when the particle moves in the jet
flow. Paleoclimatic data \cite{Dit, Dit2} also indicate such irregular
processes. There are also experimental demonstrations of L\'evy flights  in  
foraging theory and rapid geographical spread of emergent infectious
disease.   Humphries et. al. \cite{Humphries}   used GPS to track the wandering
black bowed albatrosses around an Island in Southern Indian Ocean to
study the movement patterns of searching food.   They found that
by fitting the data of the movement steps, the movement patterns
obeys the power-law property with power parameter $\alpha=1.25$.
To get the data set of human mobility that covers
all length scales,  Brockmann  \cite{Brockmann}  collected data by online bill trackers, which
  give successive spatial-temporal trajectories with a very high
resolution. When fitting the data of probability of bill traveling
at certain distances within a short period of time (less than one
week), he found power-law distribution property with power parameter $\alpha=1.6$, and
 observed that $\alpha-$stable L\'evy motions are
strikingly similar to practical data of human influenza.

L\'evy motions are thought to be appropriate models for a class of important
non-Gaussian processes with jumps \cite{Sato-99, Bertoin-98, taqqu}.
Recall that a L\'evy motion $L(t)$, or $L_t$, is a stochastic process with stationary
and independent increments. That is, for any $s, t$ with  $0\le s< t$, the distribution
of $L_t-L_s$ only depends on $t-s$, and for any $0\le t_0<t_1<\cdots<t_n$,
the random variables $L_{t_i}-L_{t_{i-1}}$, $i=1, \cdots, n$, are independent.  A L\'evy motion
has a version  whose sample paths  are almost surely
right continuous with left limits.


  Stochastic differential equations (SDEs)
with   non-Gaussian L\'evy noises
have attracted much attention  recently \cite{DuanBook2015, Apple,  Schertzer}.
To be specific, let us   consider the following  n-dimensional  \emph{state system}:
\begin{equation} \label{state}
d X_{t}  =  f (X_{t}, t)  dt +  d L_{t}^\alpha, \;\;     X_0 = x,
\end{equation}
 where $f$ is a vector field (also called a  drift),   and $L_{t}^\alpha$ is a symmetric $\alpha-$stable L\'evy motion ($0<\alpha <2$),  defined in a probability space $(\Omega, \mathcal{F}, \mathbb{P})$.
Assume that we have either   \\
(i)  a discrete time  m-dimensional \emph{observation system}:
 \begin{eqnarray} \label{observe1}
 y_k = h(x_k, t_k) + \sqrt{R_k} v_k, \;\;   k=0, 1, \cdots,
\end{eqnarray}
where $v_k$ is a white sequence of Gaussian random variables, i.e. $v_k$'s are mutually independent standard normal random variables, and $R_k$ is a sequence of nonnegative numbers; \\
or \\
(ii) a continuous time  m-dimensional  \emph{observation system}:
 \begin{equation} \label{observe2}
dY_t = h(X_t, t)dt + d W_t, \;\;      Y_0=y,
\end{equation}
where $h$ is a given vector field and $W_t$ is a Brownian motion.

\medskip

In the present paper, we estimate  system states,    with help of observations,   and in particular, we try to capture transitions between metastable states by examining most probable paths for system states.

This paper is organized as follows.  We   consider state estimates with discrete time  and continuous time observations in Sections 2 and 3, respectively.


\section{Inferring  transitions with discrete time observations }    \label{infer1}

 To demonstrate our ideas, we consider
  the following scalar SDE driven by a symmetric $\alpha$-stable L\'evy motion
\begin{equation} \label{state1}
dX_t = f(X_t, t)dt + dL_t^\alpha, \;\; X_0= x_0, 
\end{equation}
 together with observations $y_k$ are taken at discrete time instants $t_k$ as follows
\begin{eqnarray} \label{obs1}
 y_k = h(x_k, t_k) + \sqrt{R_k} v_k, \;\; k=0, 1, \cdots,
\end{eqnarray}
where $v_k$ is a white sequence of Gaussian random variables.

\medskip

For $\alpha\in (0, 2)$,  a symmetric $\alpha$-stable
  L\'evy  motion $L_t^\alpha$  has
the generating triplet   $(0, 0, \nu_\alpha)$,  where the jump measure
$$
\nu_\a({\rm d}x)=C_\alpha|x|^{-(1+\alpha)}\, {\rm d}x
$$
with $C_\alpha$ given by the formula
$\displaystyle{C_{\alpha} =
\frac{\alpha}{2^{1-\alpha}\sqrt{\pi}}
\frac{\Gamma(\frac{1+\alpha}{2})}{\Gamma(1-\frac{\alpha}{2})}}$.
For more information see \cite{Apple, DuanBook2015}.
Thus, the generator for the solution process   $X_{t} $  in \eqref{state1} is
\begin{eqnarray} \label{AABB}
A  \phi &=& ( \p_x f   )  \phi^{'}(x)     \nonumber  \\
& &+ \int_{\R^1 \setminus\{0\}} [\phi(x+  y)-\phi(x) -
I_{\{|y|<1\}} \; y \;  \phi'(x) ] \; \nu_\alpha({\rm d}y).
\end{eqnarray}
In fact, this linear operator $A$ is a nonlocal Laplace operator (\cite[Ch. 7]{DuanBook2015}), denoted also by $(-\Delta )^{\frac{\alpha}{2}}$, for $\alpha \in (0, 2)$.
The generator $A$ carries crucial information about the system state $X_t$, and hence will be useful in our investigation of state estimation.
Also note that the non-Gaussianity of the L\'evy noise manifests as nonlocality (an integral term) in the generator.
 The adjoint operator for the generator  $A$ is
 \begin{eqnarray} \label{CCDD}
 A^*p  &=&  -\p_x (f(x, t) p(x, t))     \nonumber \\
& &+   \int_{\R^1 \setminus\{0\}}
 [p(x+y, t)-p(x,t) - I_{\{|y|<1\}} \; y \; \p_x p(x, t) ] \; \nu_\alpha(dy).
\end{eqnarray}

The Fokker-Planck equation  for the SDE (\ref{state1}) is  (\cite{Apple, DuanBook2015}):
\begin{eqnarray} \label{fpe-2999}
p_t &=& -\p_x (f(x, t) p(x, t))     \nonumber \\
& &+   \int_{\R^1 \setminus\{0\}} [p(x+y, t)-p(x,t) - I_{\{|y|<1\}} \; y \p_x p(x, t) ] \; \nu_\alpha(dy).
\end{eqnarray}

\medskip


Denote $Y_{t_k} = \{ y_0, y_1,\cdots, y_k\}$.   Similarly as in \cite{Jaz}, we have the following theorem which determines the time evolution of the conditional probability density function $p(x,t\mid Y_t)$. For convenience, we often write $p(x, t)$ for $p(x,t\mid Y_t)$.

\begin{theorem}(Conditional Density for Continuous-discrete
Problems). Let system (\ref{state1}) satisfy the hypotheses that $f$ is
Lipschitz in  space and  the initial state $X_0$, with  the property
$E(|X_0|^2)<\infty$,  is independent of $\{L_t^\alpha, t\in[t_0,T]\}$. Suppose
that the prior density $p(x,t)$ for(\ref{state1}) exists and is once
continuously differentiable with respect to $t$ and twice with
respect to $x$. Let $h$ be continuous in both arguments and bounded
for each $t_k$   with probability  $1$. \\
Then, between observations, the conditional density $p(x,t|Y_t)$
satisfies the Fokker-Planck equation
\begin{equation}\label{Kolmog}
dp(x,t|Y_t) = A^*p \;  dt, \ \ \ t_k\leq t < t_{k+1},\ \
p(x,t_0|Y_{t_0}) = p(x_{t_0}),
\end{equation}
where $A^*$ is the   operator in (\ref{CCDD}). At an observation
$(\text{at }  t_k) $,    the conditional density satisfies the  following difference
equation
\begin{eqnarray}
p(x, t_k|Y_{t_k}) = \frac{p(y_k|x)p(x, t_k|Y_{t_k^-})
}{\int_{\R^1}p(y_k|\xi) p(\xi, t_k|Y_{t_k^-})d\xi },
\end{eqnarray}
where $p(y_k|x)$ is
\begin{equation}
p(y_k|x) = (1/(2\pi)^{1/2}|R_k|^{1/2}) exp\{ - \frac{1}{2}
[y_k-h(x,t_k)]^T R_k^{-1}[y_k-h(x,t_k)]\}.
\end{equation}

\end{theorem}

\begin{proof}
The conditional density in the absence of observation, satisfies
the Fokker-Planck equation.  Therefore, between observations,
conditional density $p(x,t|Y_t)$ satisfies  the Fokker-Planck
equation  (\ref{Kolmog}).

Thus, it remains to determine the relationship between
$p(x,t_k|Y_{t_k})$ and $$p(x,t_k|Y_{t_k^-}) \equiv
p(x,t_k|Y_{t_{k-1}}).$$

Since $p(x,t_k|Y_{t_k}) = p(x,t_k|y_k,Y_{t_{k-1}})$, we have by
Bayes' rule
$$p(x,t_k|Y_{t_k}) = \frac{p(y_k|x,t_k,Y_{t_{k-1}} ) P(x,t_k|Y_{t_{k-1}})}{p(y_k|Y_{t_{k-1}})}=
\frac{p(y_k|x_k,Y_{t_{k-1}} )
P(x,t_k|Y_{t_{k-1}})}{p(y_k|Y_{t_{k-1}})}.$$ 
Now, since the noise
$\{v_k\}$ is white, $$p(y_k|x_k, Y_{t_{k-1}})=p(y_k|x_k).$$
Similarly, we compute 
$$ p(y_k|Y_{t_{k-1}}) = \int
p(y_k|x) p(x,t_k|Y_{t_{k-1}})dx.$$ Therefore,
$$ p(x, t_k|Y_{t_k}) = \frac{p(y_k|x)p(x, t_k|Y_{t_k^-}) }{\int_{\R^1}p(y_k|\xi) p(\xi, t_k|Y_{t_k^-})d\xi }.$$
This completes the proof.
\end{proof}

\bigskip




This theorem provides the foundation for computing conditional density for system state of SDE (\ref{state1}), under discrete time observations.

Define $x^*(x_0, t) \triangleq \mbox{ maximizer for } \max_{x\in \R^1} p(x, t)$.\\
This provides the  most probable  orbit  (\cite{DuanBook2015, ChengDuanWang}) starting at $x_0$. These most probable orbits are the maximal likely orbits for a dynamical system under noisy fluctuations.  

Let us consider an example.

\begin{example} \label{discrete}
Let us consider a scalar system with state equation
\begin{equation} \label{sde}
{\rm d}X_{t}  = 4(X_t - X_t^3)  {\rm d}t +  \sqrt{0.24} \; {\rm d}
L_{t}^\alpha, \;\; X_0 = x.
\end{equation}
The discrete-time scalar observation is:
\begin{eqnarray} \label{obs1}
 y_k = h(x_k, t) + \sqrt{R_k} v_k,
\end{eqnarray}
with $h(x, t) = x$,  and $ R_k   \equiv 0.1$.

In the absence of L\'evy  noise,  this system has two stable states: $-1$ and $+1$.
When the   noise kicks in, these two states are no longer fixed. The random system evolution near these two states, together with possible transitions between them, is sometimes called a metastable phenomenon \cite{meta}.  For convenience, we call $-1$ and $+1$ (and random motions   nearby)    metastable states.

The corresponding nonlocal Fokker-Planck equation is computed on $(-2.5,2.5)$
with a finite difference scheme (\cite{Gao, Gao2}) under  the natural boundary condition. Space stepsize $\Delta x=0.05$ and time
stepsize $\Delta t= 0.001$.  The initial probability density is taken either as a Gaussian distribution or a uniform distribution.

  In Figures \ref{discrete1} and  \ref{discrete2}, we show the conditional density $p(x, t)$, together with the corresponding most probable orbit (taken   as the state estimation  for $X_t$),  together with observations $y_k$ and a true state path $X_t$. The initial density is either Gaussian or uniform, but centered at the metastable state $-1$. Notice that the estimated state captures the transitions from $-1$ to $+1$ and then back to $-1$, during the time period $0 <t<10$.

\begin{figure}[]
\begin{center}
\includegraphics*[width=6.2cm,height=5cm]{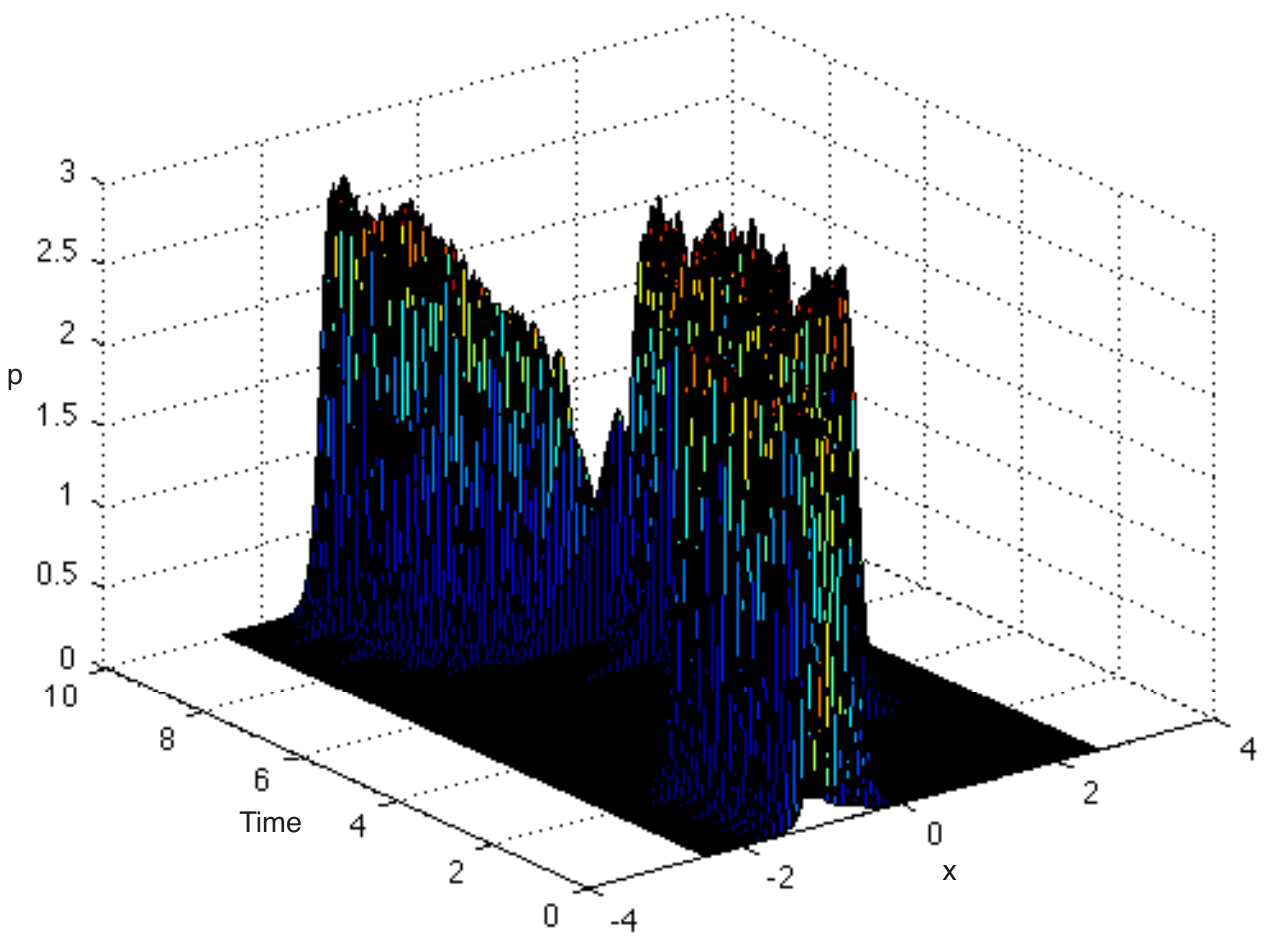}
\includegraphics*[width=6.2cm,height=5cm]{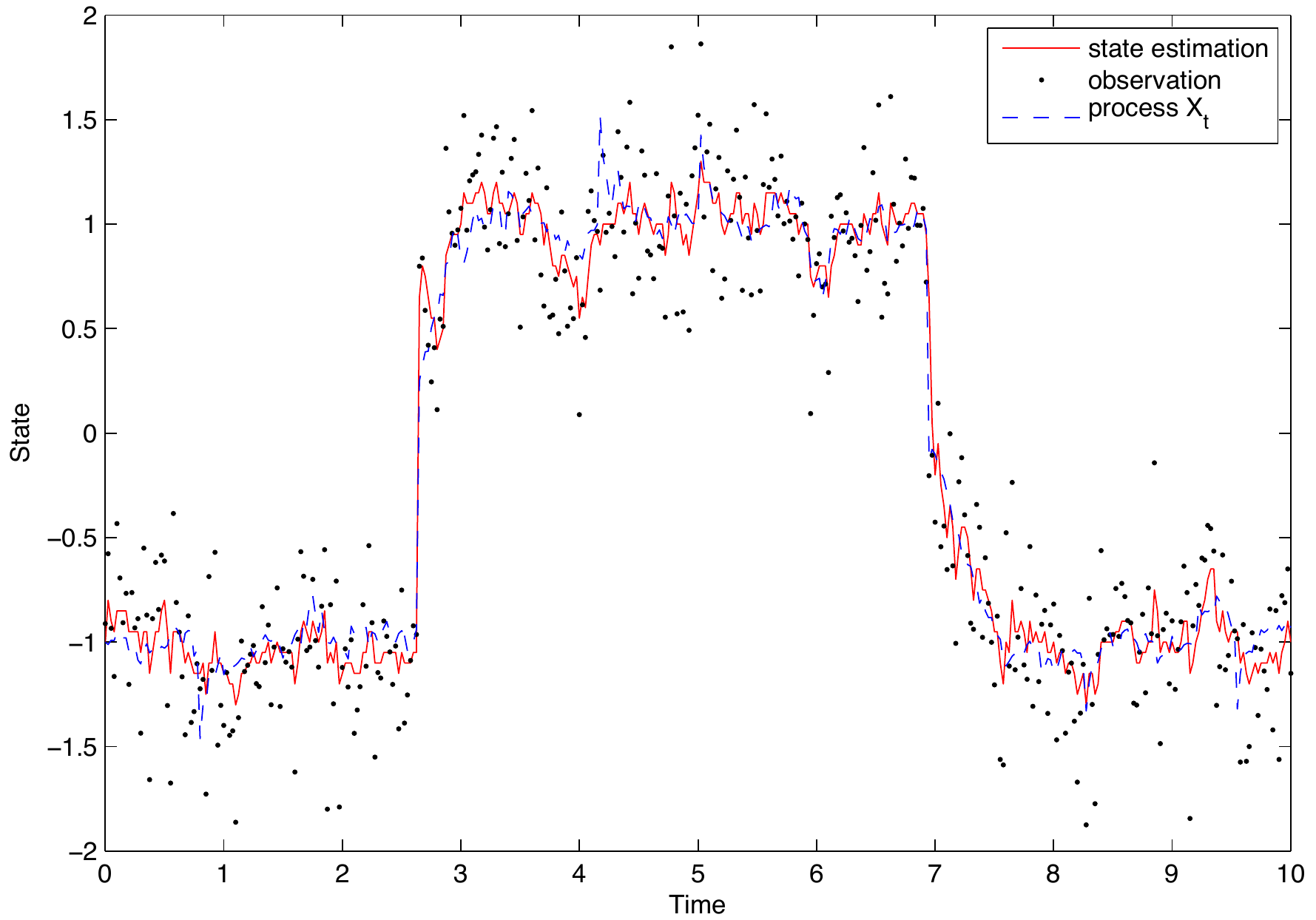}
\end{center}
\caption{Example \ref{discrete} --  Conditional probability and state estimation when
$\alpha=1.5$. The initial probability density is Gaussian centered
at $-1$.  Online version:  The red curve in the right panel is the estimate state orbit (i.e., most probable orbit).  }
\label{discrete1}
\end{figure}

\begin{figure}[]
\begin{center}
\includegraphics*[width=6.2cm,height=5cm]{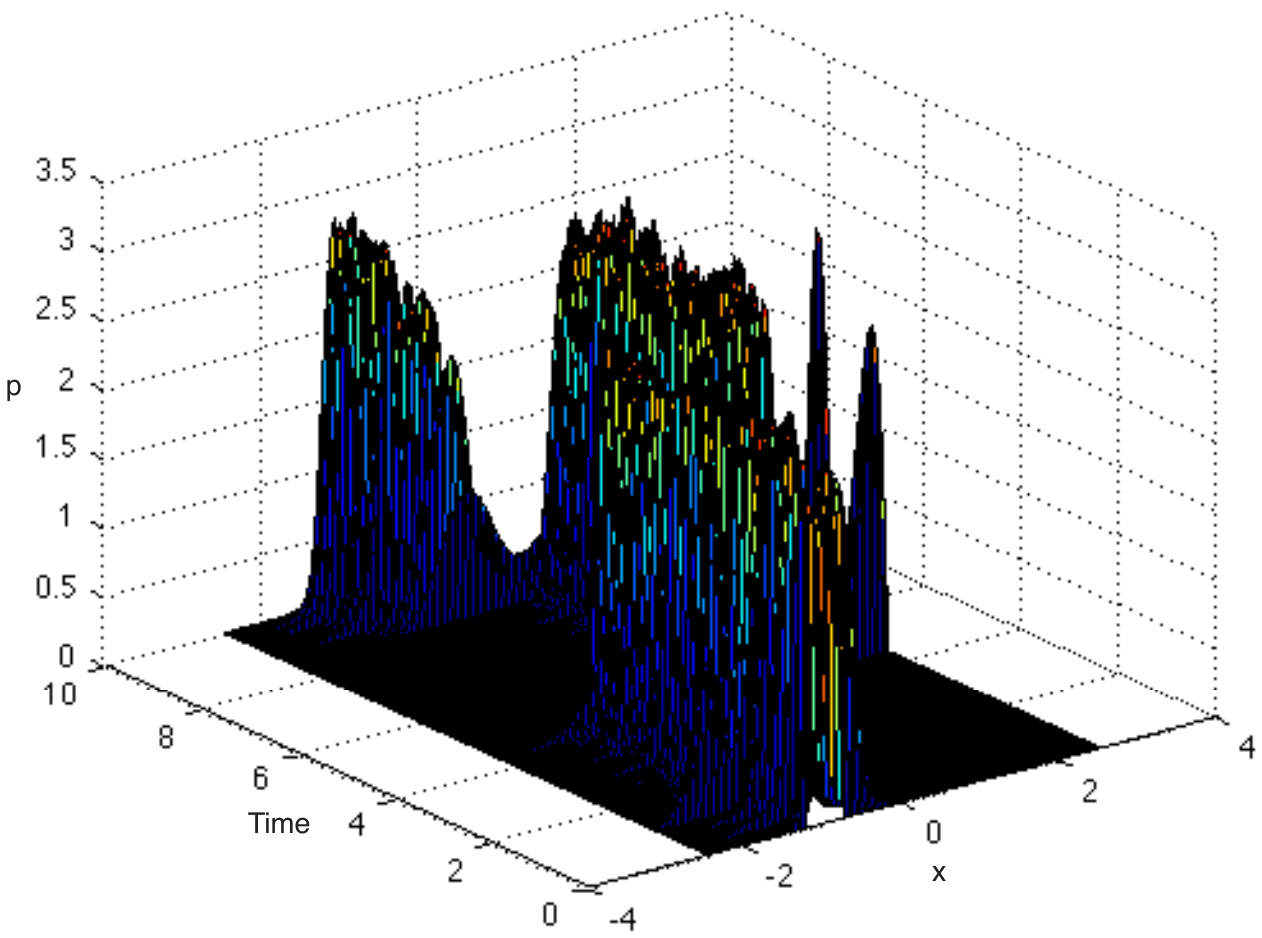}
\includegraphics*[width=6.2cm,height=5cm]{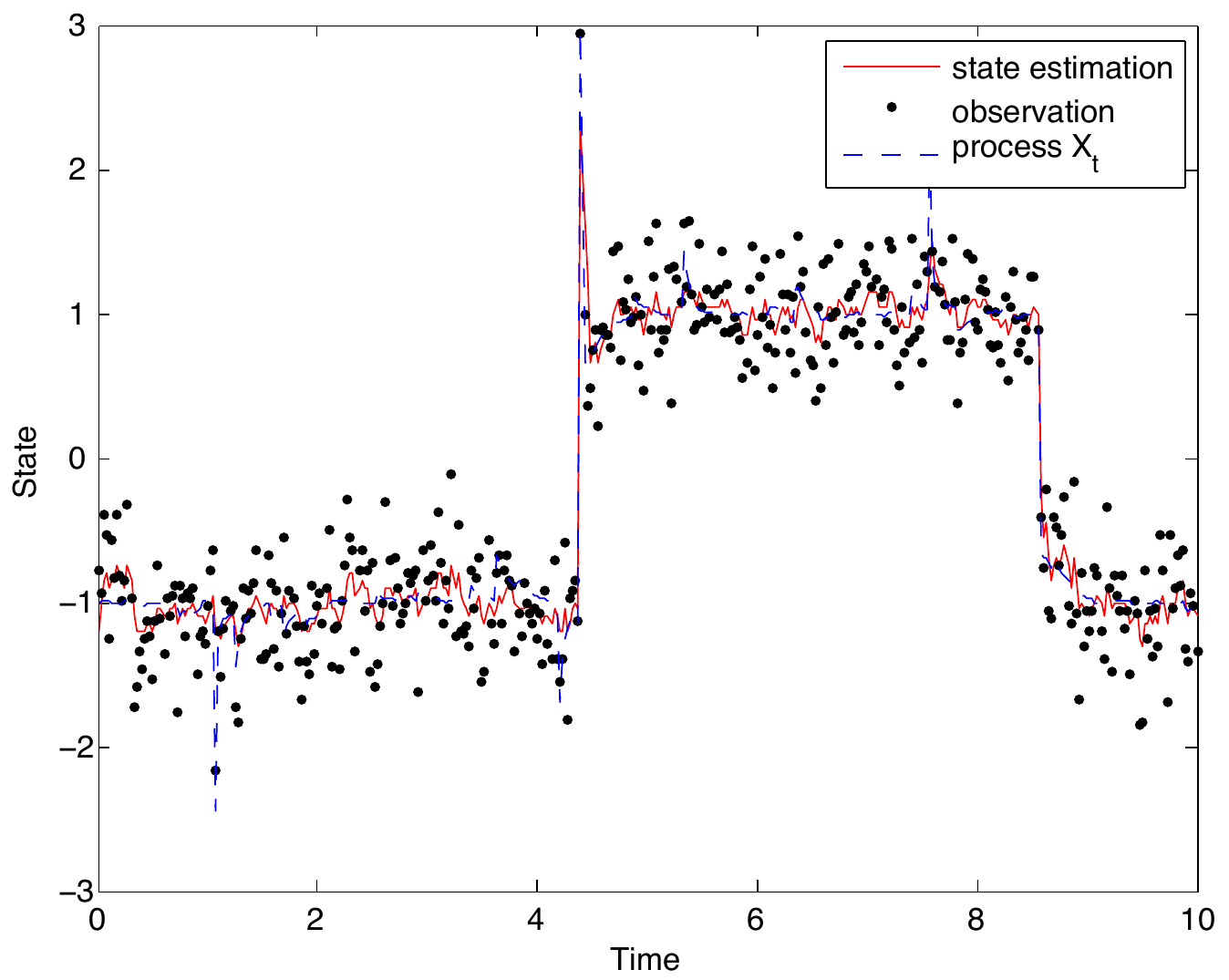}
\end{center}
\caption{Example \ref{discrete} --   Conditional probability and state estimation when
$\alpha=0.75$. The initial probability density is uniform on
$(-1.25,-0.75)$.   Online version:  The red curve in the right panel is the estimate state orbit (i.e., most probable orbit).   }
\label{discrete2}
\end{figure}

\end{example}



\section{Inferring transitions with continuous time observations} \label{continuous}

We consider the following  scalar  \emph{state system}
with a  symmetric $\alpha-$stable  L\'evy motion
\begin{equation} \label{state3}
d X_{t}  =  f (X_{t}, t)  dt +  d L_{t}^\alpha, \;\;     X_0 = x,
\end{equation}
together with
  a continuous time  scalar  \emph{observation system}:
 \begin{equation} \label{observe3}
dY_t = h(X_t, t)dt + d W_t, \;\;      Y_0=y,
\end{equation}
where $h$ is a given vector field and $W_t$ is a Brownian motion.

The unnormalized conditional probability density
$p(x, t|Y)$ satisfies a  nonlocal Zakai equation  ( \cite{Qiao, Miku, Popa}):

\begin{eqnarray} \label{Zakai}
dp(x, t|Y) = A^*p(x, t|Y) \; dt + h(x, t) p(x, t|Y)  \; dY_t,
\end{eqnarray}
where $A^*$ is the adjoint operator of the generator $A$:
\begin{eqnarray} \label{AABBCC}
A^*  \phi &=& - (f(x, t) \phi(x))^{'}     \nonumber  \\
& &+ \int_{\R^1 \setminus\{0\}} [\phi(x+  y)-\phi(x) -
I_{\{|y|<1\}} \; y \phi'(x) ] \; \nu_\alpha({\rm d}y),
\end{eqnarray}
and  $p(x, 0|Y)$ is the initial   density of $X_t$ (say a uniform distribution near the metastable state $-1$). The Zakai equation \eqref{Zakai} may be numerically solved with a finite difference method based on \cite{Gao, Gao2} together with a discretization of the noisy term at the current space-time point and $dY_t  \approx Y_{t+\Delta t} - Y_t$.  The initial probability density is taken either as a Gaussian distribution or a uniform distribution.    For other numerical methods, see, for example,  \cite{Loto, Crisan, Law, Cai, Yau}.

\begin{remark}
The normalized conditional probability density  $p(x,t\mid Y_t)$ satisfies the nonlinear Kushner's equation
\begin{align*}
dp(x,t\mid Y_t) &= A^*p(x,t\mid Y_t)dt + (h(x,t) - \hat h(x_t,t))(dz_t - \hat h(x,t)dt)p(x,t\mid Y_t)\\
&= A^*p(x,t\mid Y_t)dt + h(x,t)p(x,t\mid Y_t)dz_t - h\hat hpdt -\hat hpdz_t + \hat h\hat h pdt,
\end{align*}
where $\hat h(x_t,t)  $ is the mathematical expectation of $h(X_t, t)$, with respect to $p$.
 \end{remark}



The conditional density  $p(x, t|Y)$ provides information for the system evolution.
With the observation,    we can infer  possible transitions from the metastable state $-1$ to the metastable state $+1$, within a  time range $(0, T)$.
If the system starts with a   probability distribution $p_0(x)$ near  the metastable state $x=-1$,  then   the conditional density  $p(x, t|Y)$  helps us to infer whether  the system will get near the other metastable state $+1$, and vice versa.
This may be achieved by  examining  the most probable orbits for the system, under the observation.
Define $x^*(x_0, t) \triangleq \mbox{ maximizer for } \max_{x\in \R^1} p(x, t)$.\\
This provides the  most probable  orbit  (\cite{DuanBook2015, ChengDuanWang}) starting at $x_0$. We take this as our state estimation  for $X_t$, as in \cite{Miller}.



Let us illustrate this by an example.

\begin{example}  \label{continuous}
Let us consider the following scalar SDE state equation with a  symmetric $\alpha-$stable  L\'evy motion:
$$
dX_t=4(X_t-X_t^3)dt + \sqrt{0.24}\; d L^\alpha, \;\;     X_0=x_0.
$$
The scalar observation   equation is given by
$$
dY_t = X_tdt + \sqrt{0.05} \; dW_t
$$
When noise is absent, the state system has two stable states: $-1$ and $+1$.


The corresponding nonlocal Zakai equation is solved by a finite difference scheme with 
space stepsize $\Delta x=0.05$ and time
stepsize $\Delta t= 0.001$.  

In Figures (\ref{continuous1}) - (\ref{continuous2}), we show the conditional density $p(x, t)$, together with the corresponding most probable orbit (taken   as the state estimation  for $X_t$),  together with  a true state path $X_t$. The initial density is either Gaussian or uniform, but centered at the metastable state $-1$. Notice that the estimated state captures the transitions from $-1$ to $+1$, then back to $-1$, and finally to $+1$, during the time period $0 <t<50$.

\begin{figure}[]
\begin{center}
\includegraphics*[width=6.2cm,height=5cm]{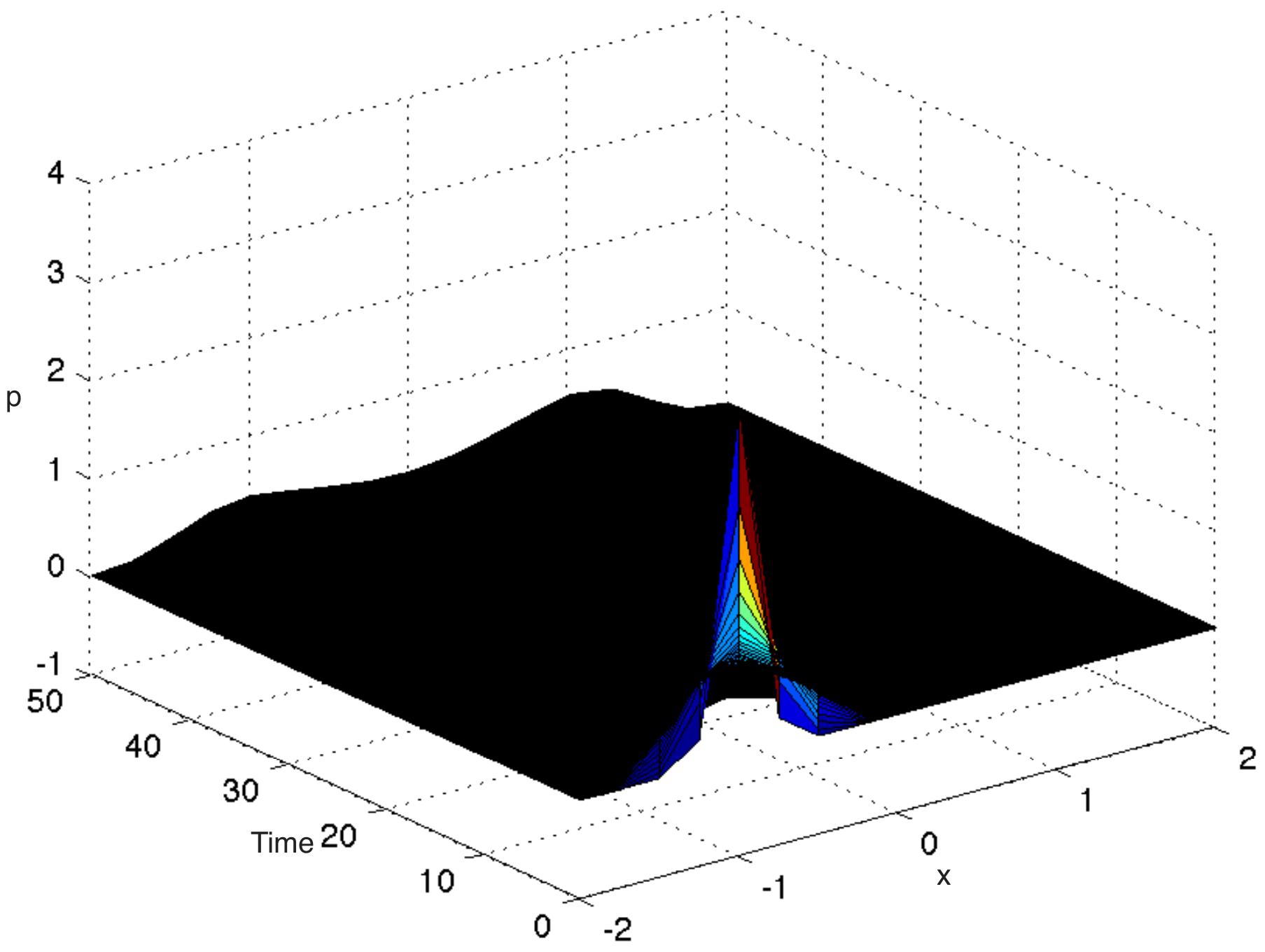}
\includegraphics*[width=6.2cm,height=5cm]{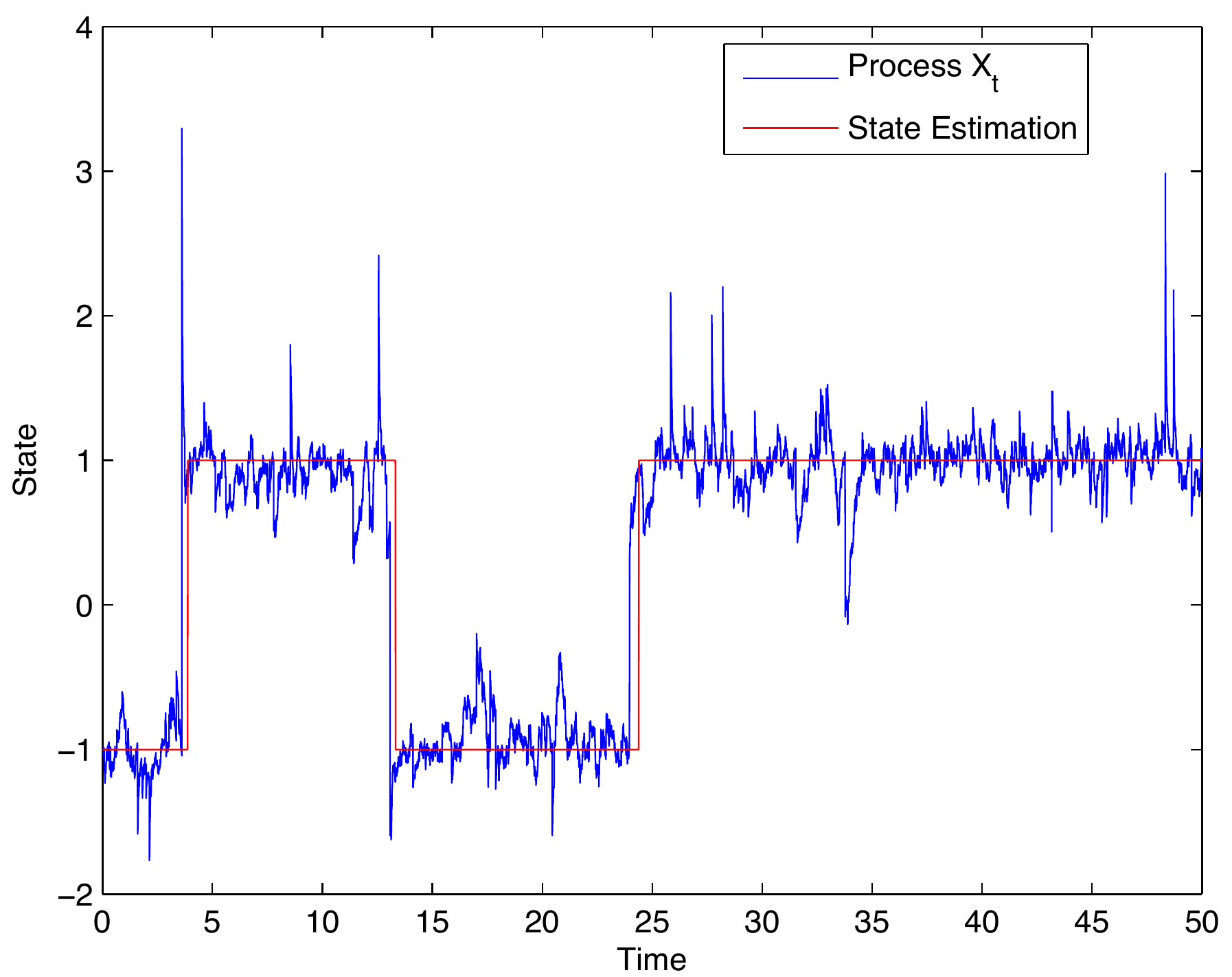}
\end{center}
\caption{Example \ref{continuous} --  Conditional probability (left) and state estimation (right) when
$\alpha=1.5$. The initial probability density is Gaussian centered
at $-1$.}
\label{continuous1}
\end{figure}

\begin{figure}[]
\begin{center}
\includegraphics*[width=6.2cm,height=5cm]{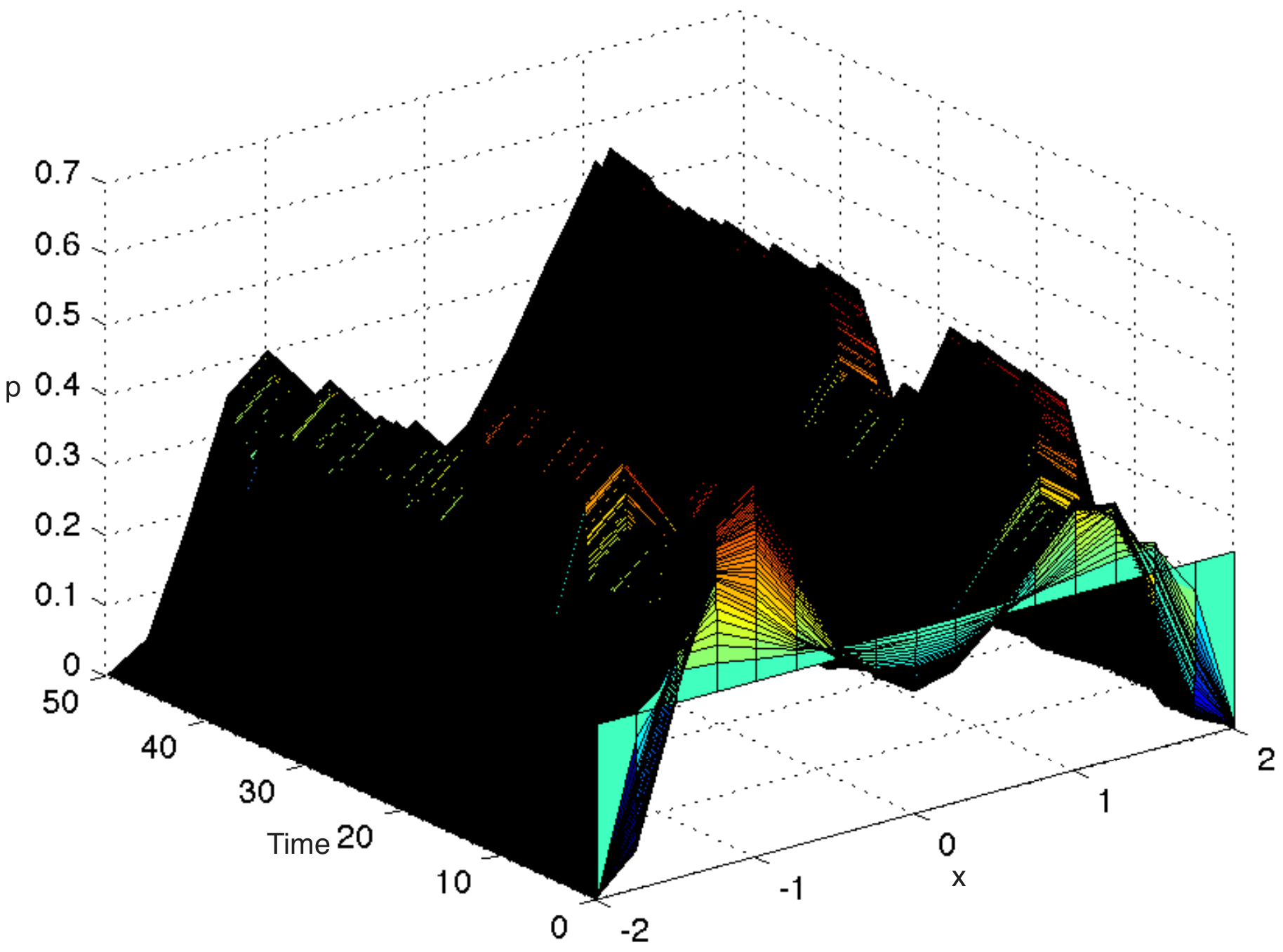}
\includegraphics*[width=6.2cm,height=5cm]{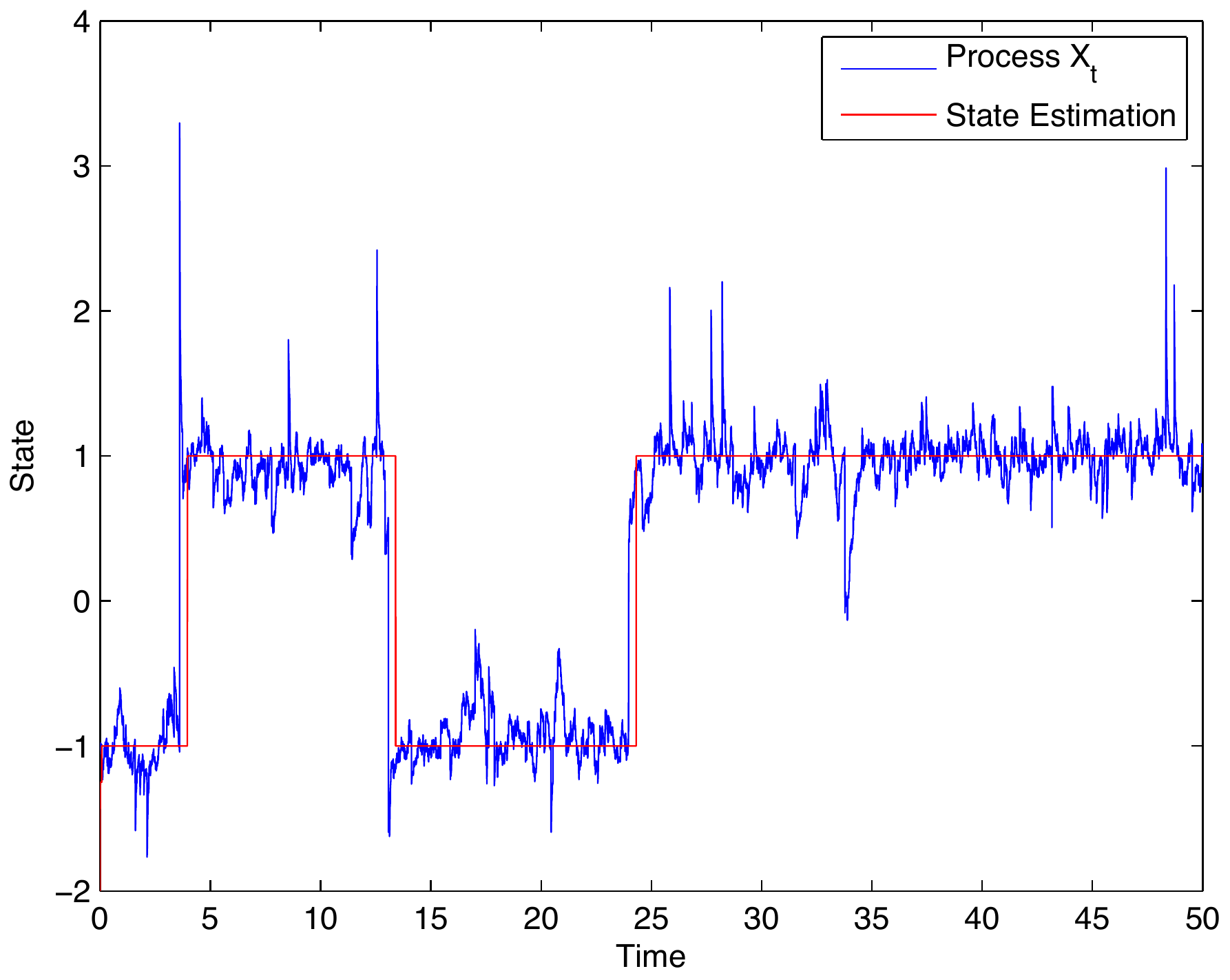}
\end{center}
\caption{Example \ref{continuous} --  Conditional probability (left) and state estimation (right) when
$\alpha=1.5$. The initial probability density is uniform on $(-2,2)$.}
\label{continuous2}
\end{figure}

\end{example}


\end{document}